\documentclass[12pt]{article}
\usepackage{amsfonts}
\usepackage{amsmath}
\usepackage{tikz}
\usetikzlibrary{matrix,arrows}
\usetikzlibrary{through}

\newcommand\TT{\rule{0pt}{5ex}}
\newcommand\T{\rule{0pt}{3ex}}

\newcommand\BB{\rule[-3ex]{0pt}{0pt}}

\newtheorem{theorem}{Theorem}

\newtheorem{example}{Example}

\newtheorem{proposition}[theorem]{Proposition}

\newenvironment{proof}[1][Proof]{\begin{trivlist}
\item[\hskip \labelsep {\bfseries #1}]}{\end{trivlist}}

\begin{document}

\title{Computing nonsymmetric and interpolation Macdonald polynomials}
\author{W. Baratta \\
	Department of Mathematics, University of Melbourne}
	\maketitle

\begin{abstract}
In this paper we present a Mathematica notebook for computing nonsymmetric and interpolation Macdonald polynomials. We present the new recursive generation algorithm employed within the notebook and the theory required for its development. We detail the contents of the notebook and conclude with a couple of applications of the notebook.
\end{abstract}

\section{Introduction\label{introduction}}

The nonsymmetric Macdonald polynomial $E_\eta(z;q,t)$ and its generalisation - the interpolation Macdonald polynomial $E^*_\eta(z;q,t)$ - have found applications in mathematical physics, combinatorics and representation theory \cite{haiman, haglund, bernevig, kakei}.

A feature of the aforementioned polynomials is that they allow for explicit computation. Having this knowledge provides an opportunity to experimentally seek new properties and to check analytic work. Within this paper a computer software program that was developed for these purposes is detailed.

The nonsymmetric and interpolation Macdonald polynomials are generated via recursive generation formulas. This paper begins with the theory required to introduce these formulas (Section \ref{foundation}) and then later shows how these formulas can be used to construct a recursive generation algorithm for the generation of the nonsymmetric and interpolation polynomials (Proposition \ref{recgenpoly}). Section \ref{Mathematicanotebook} details the functions within the Mathematica notebook containing this algorithm and the paper concludes with two possible applications of the notebook.

\section{Foundation Theory}\label{foundation}
\subsection{Compositions}
The polynomials discussed in this paper are labelled by compositions. We define a composition to be an $n$-tuple $\eta:=(\eta_1,\ldots,\eta_n)$ of non-negative integers. Each $\eta_i$ is called a component, the length of a composition is the number of components it contains and the sum of the components is called the modulus and denoted by $|\eta|$. Each composition $\eta$ corresponds to a unique partition $\eta^+$ obtained by rearranging the components of $\eta$ so that they are nondecreasing.

Two important operators that act on compositions are the switching and raising operators. We have the switching operator $s_i$ acting on compositions according to $$
s_i(\eta_1,\ldots, \eta_i,\eta_{i+1},\ldots,\eta_n):=(\eta_1,\ldots,\eta_{i+1},\eta_i,\ldots,\eta_n),\hspace{1cm}i=1,\ldots,n-1
$$
and the raising operator $\Phi$ which has the action
\begin{equation}\label{raising}
\Phi \eta:=(\eta_2,\ldots,\eta_n,\eta_1+1).
\end{equation}

It can be seen that every composition $\eta$ can be recursively generated from the all zero composition $(0,\ldots,0)$. In Section \ref{recgencompsection} we construct an algorithm to recursively generate any composition from $\eta$ from $(0,\ldots,0)$ using the least number of operators. It is this algorithm that allows us to obtain our main result, the recursive generation algorithm for the nonsymmetric and interpolation polynomials (Proposition \ref{recgenpoly}).


We now proceed to the required polynomial theory.

\subsection{Polynomial theory}

The nonsymmetric Macdonald polynomial \cite{cherednik} and interpolation Macdonald polynomials \cite{knop} are most commonly defined by their eigenfunction and vanishing properties, respectively. Here we take an alternative approach and define them via their respective recursive generation operators.

In the previous section we noted that every composition $\eta$ can be recursively generated from $(0,\ldots,0)$ using the operators $s_i$ and $\Phi$. In a similar way the nonsymmetric Macdonald polynomial $E_\eta(z;q,t):=E_\eta(z)$ and the interpolation Macdonald polynomial $E^*_\eta(z;q,t):=E_\eta^*(z)$ can be recursively generated from $E_{(0,\ldots,0)}(z;q,t)=1$ and $E_{(0,\ldots,0)}^*(z;q,t)=1$, respectively, using two elementary operators. We begin by introducing the switching-type operators for the two polynomials; namely the Demazure-Lustig operator $T_i$ for the nonsymmetric polynomials and the Hecke operator $H_i$ for the interpolation polynomials. These operators relate $E_\eta(z)$ and $E_{s_i\eta}(z)$, and $E_\eta^*(z)$ and $E^*_{s_i\eta}(z)$, respectively.

Each switching-type operator is a realisation of the type-$A$ Hecke algebra, an associative unitaal algebra over $\mathbb{Q}(t)$ generated by elements $h_1,\ldots,h_{n-1}$ and subject to the relations
\begin{align}
h_{i}h_{i+1}h_{i}&=h_{i+1}h_{i}h_{i+1} , \hspace{1cm}\text{for }
1\leq i\leq n-2 \notag\\
h_{i}h_{j}&=h_{j}h_{i},\text{ }\left\vert i-j\right\vert >1\text{ }\label{hecke algebra} \\
( h_{i}+1) ( h_{i}-t) &=0.  \notag
\end{align}

The switching type operators are defined by
\begin{equation}\label{TiHi}
T_i:=t+\frac{tz_i-z_{i+1}}{z_i-z_{i+1}}(s_i-1)\hspace{0.5cm} \text{   and   } \hspace{0.5cm} H_i:=t+\frac{z_i-tz_{i+1}}{z_i-z_{i+1}}(s_i-1),
\end{equation}
where here, the switching operator $s_i$ is defined to act on functions according to
$$
(s_i f)(z_1,\ldots,z_i,z_{i+1},\ldots,z_n):=f(z_1,\ldots,z_{i+1},z_{i},\ldots,z_n).
$$

These operators act on their respective polynomials according to \cite{mimachi}
\begin{equation}
T_{i}E_{\eta }(z)=\left\{
\begin{tabular}{ll}
$ \frac{t-1}{1-\delta _{i,\eta }^{-1}( q,t) }
E_{\eta }(z)+tE_{s_{i}\eta }(z)$, & $\eta _{i}<\eta _{i+1}$ \\
$tE_{\eta }(z)$, & $\eta _{i}=\eta _{i+1}$ \\
$ \frac{t-1}{1-\delta _{i,\eta }^{-1}( q,t) }
E_{\eta }(z)+\frac{( 1-t\delta _{i,\eta }( q,t) ) (
1-t^{-1}\delta _{i,\eta }( q,t) ) }{( 1-\delta
_{i,\eta }( q,t) ) ^{2}}E_{s_{i}\eta }(z)$, & $\eta _{i}>\eta
_{i+1}$,
\end{tabular}
\right.  \label{TiEn}
\end{equation}
and \cite{sahiinterpolation,knop}
\begin{align}
H_{i}E_{\eta }^{\ast }( z) =&\left\{
\begin{tabular}{ll}
$\frac{t-1}{1-\delta _{i,\eta }^{-1}( q,t) }E_{\eta }^{\ast
}( z) +E_{s_{i}\eta }^{\ast }( z) $, & $\eta _{i}<\eta
_{i+1}$ \\
$tE_{\eta }^{\ast }( z) $ ,& $\eta _{i}=\eta _{i+1}$ \\
$\frac{t-1}{1-\delta _{i,\eta }^{-1}( q,t) }E_{\eta }^{\ast
}( z) +\frac{( 1-t\delta _{i,\eta
}( q,t)) (t- \delta _{i,\eta }( q,t)) }{( 1-\delta
_{i,\eta }( q,t)) ^{2}}E_{s_{i}\eta }^{\ast
}( z) $, & $\eta _{i}>\eta _{i+1}$.
\end{tabular}
\right.  \label{hecke on interpolation}
\end{align}
In (\ref{TiEn}) and (\ref{hecke on interpolation}) $\delta _{i,\eta }( q,t) :=\overline{\eta }_i/\overline{
\eta} _{i+1}$, with
$$
\overline{\eta}_i:=q^{\eta_i}t^{-l'_\eta(i)},
$$
where
$$
l_\eta'(i):=\#\{j<i;\eta_j\geq \eta_i\}+\#\{j>i;\eta_j>\eta_i\}.
$$
The raising type-operators $\Phi_q$ and $\Phi_q^*$ of the nonsymmetric and interpolation polynomials, respectively, are given by \cite{peter3}
\begin{equation}
\Phi_q:=z_nT_{n-1}^{-1}\ldots T_1^{-1}=t_{i-n}T_{n-1}\ldots T_i z_i T_{i-1}^{-1}\ldots T_i^{-1} \label{NSRaising}
\end{equation}
and
\begin{equation}\label{IntRaising}
\Phi_q^*:=(z_n-t^{-n+1})\Delta.
\end{equation}
In (\ref{NSRaising}) the operator $T_i^{-1}$ is related to $T_i$ by the quadratic relation in (\ref{hecke algebra}) and given explicitly by
$$
T_i^{-1}:=t^{-1}-1+t^{-1}T_i,
$$
and in (\ref{IntRaising}) $\Delta f(z_1,\ldots,z_n):=f(z_n/q,z_1,\ldots,z_{n-1})$.

The raising type operators act on the Macdonald polynomials according to \cite{peter3}
$$
\Phi_q E_\eta(z)=t^{-\#\{i>1;\eta_i\geq \eta_1\}}E_{\Phi \eta}(z)
$$
and \cite{knop}
$$
\Phi_q^*E_\eta^*(z)=q^{-\eta_1}E^*_{\Phi\eta}(z).
$$

In Section \ref{Generating polynomials} we provide details of the algorithm that implements these formulas to recursively generate any nonsymmetric or interpolation Macdonald polynomial.

Below are some examples on nonsymmetric and interpolation polynomials, for further details on the polynomials we refer the reader to \cite{barattathesis}.

\begin{align*}
E_{(0,3)}(z)&=z_2^3+\tfrac{t-1}{q^2t-1}z_1^2z_2^1+\tfrac{(q+1)(t-1)}{q^2t-1}z_1^1z_2^2\\
E_{(2,1)}(z)&=z_1^2z_2^1+\tfrac{q(t-1)}{qt-1}z_1^1z_2^2\\
E_{(1,2)}(z)&=z_1^1z_2^2\\
\\
E^*_{(1,1)}(z)&=z_1z_2-\tfrac{1}{t}z_1-\tfrac{1}{t}z_2+\tfrac{1}{t^2}\\
E^*_{(1,0)}(z)&=z_1+\tfrac{t-1}{qt-1}z_2-\tfrac{q t^2-1}{t (q t-1)}\\
E^*_{(0,1)}(z)&=z_2-\tfrac{1}{t}.
\end{align*}

Before considering the recursive generation algorithm we discuss the limiting properties of the nonsymmetric and interpolation Macdonald polynomials.

\subsection{Limiting properties}

The following diagram indicates that the nonsymmetric Macdonald polynomials can be obtained from the interpolation Macdonald polynomials through a process of homogenisation, that is taking only the terms of top degree from the interpolation polynomial extracts the nonsymmetric polynomial. Furthermore, the diagram indicates that there are many additional families of polynomials that can be obtained from the nonsymmetric and interpolation Macdonald polynomials. It is this feature that allows the computational work discussed in this paper to immediately yield algorithms for the generation of these other families of polynomials.

Through symmetrisation the nonsymmetric Macdonald polynomial $E_\eta(z;q,t)$ can be reduced to symmetric Macdonald polynomial $P_{\eta^+}(z;q,t)$. Additionally, the symmetric, nonsymmetric and interpolation Macdonald polynomials reduce to the corresponding Jack polynomials, $P_{\eta^+}(z;\alpha)$, $E_\eta(z;\alpha)$, $E_\eta^*(z;\alpha)$, in the limit $t=q^{1/\alpha}$, $q\rightarrow 1$, which themselves reduce to the Schur and zonal polynomials, $s_{\eta^+}(z)$, $Z_{\eta^+}(z)$, by setting $\alpha$ to $1$ and $2$, respectively. Lastly, we have that the symmetric Macdonald polynomials reduce to the Hall-Littlewood polynomials $P_{\eta^+}(z;t)$ in the limit $q=0$.

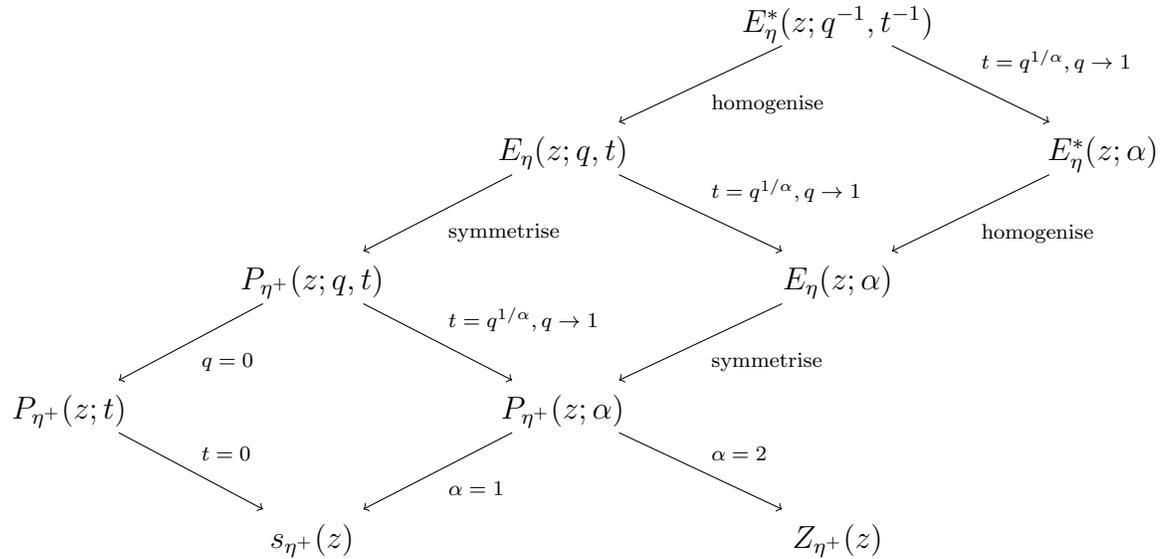
\begin{center}
\begin{tikzpicture}[description/.style={fill=white,inner sep=2pt}]
\matrix (m) [matrix of math nodes, row sep=2.5em,
column sep=3em, text height=2ex, text depth=0.25ex]
{&&  &E^*_\eta(z;q^{-1},t^{-1})&\\
&& E_\eta(z;q,t) &&E^*_\eta(z;\alpha)\\
& P_{\eta^+}(z;q,t) & & E_\eta(z;\alpha)\\
 P_{\eta^+}(z;t)&  &P_{\eta^+}(z;\alpha)& &\\
& s_{\eta^+}(z) & &Z_{\eta^+}(z) & \\};
\path[->,font=\scriptsize]
(m-1-4) edge node[auto] { homogenise} (m-2-3)
(m-2-5) edge node[auto] { homogenise} (m-3-4)
(m-2-3) edge node[auto] { symmetrise} (m-3-2)
(m-3-4) edge node[auto] { symmetrise} (m-4-3)
(m-1-4) edge node[auto] { $t=q^{1/\alpha},q\rightarrow1$} (m-2-5)
(m-2-3) edge node[auto] { $t=q^{1/\alpha},q\rightarrow1$} (m-3-4)
(m-3-2) edge node[auto] { $t=q^{1/\alpha},q\rightarrow1$} (m-4-3)
(m-3-2) edge node[auto] { $q=0$} (m-4-1)
(m-4-1) edge node[auto] { $t=0$} (m-5-2)
(m-4-3) edge node[auto] { $\alpha=1$} (m-5-2)
(m-4-3) edge node[auto] {$ \alpha=2$ } (m-5-4);
\end{tikzpicture}
\end{center}

The definitive reference for the symmetric Macdonald polynomials is \cite{macdonald}. This book also contains details on the Schur and Hall-Littlewood polynomials. An accessible reference for the Jack polynomials and zonal polynomials are \cite{forresterbook} and \cite{muirhead}, respectively. For details on how the symmetric Macdonald polynomials can be obtained from the nonsymmetric polynomials we refer the reader to \cite{marshallmacdonald}.

\section{Computer Programming}

\subsection{A brief history}

Computer software that generates the polynomial families listed above is quite limited. One of the few programs available is designed to compute symmetric Jack polynomials \cite{mops}. The software, written in Maple, uses recurrence formulas to evaluate the coefficients $K'_{\kappa\mu}(\alpha)\in\mathbb{Q}(\alpha)$ in the expansion
\begin{equation*}\label{CJack}
P_{\eta^+}(z;\alpha)=\sum_{\lambda^+\leq\eta^+}K'_{\lambda^+\eta^+}(\alpha)m_{\lambda^+}(z),
\end{equation*}
where the $m_{\lambda^+}(z)$ are  symmetric monomial functions (see for example \cite{macdonald}). The study in \cite{mops} discusses the implementation of the software and the run times of particular functions. They also provide details for the generation of the generalised classical Hermite, Laguerre and Jacobi polynomials.

A further study in the generation of symmetric Jack polynomials is by Demmel and Koev \cite{koev}. They use the expansion formula for a Jack polynomial in one of its variables to obtain a more efficient evaluation for sets of Jack polynomials than was known previously.

Although there are known methods for generating the nonsymmetric and interpolation Macdonald polynomials, for example the Rodrigues formulas \cite{nishinobc}, software that generates nonsymmetric and interpolation polynomials appears to be nonexistent in the literature. It is our aim to initiate momentum in this area.

We note that the methods presented here would be dramatically improved in efficiency if the algorithms introduced by Demmel and Koev were generalisable to nonsymmetric theory. Although progress has been made using computer-generated coefficients, confirming the ability to generalise the algorithms, the formulas cannot be run efficiently until the nonsymmetric equivalent of equation (6.2) of \cite{koev} --- a type of dual Pieri formula --- is known.

We now present the algorithms implemented in our Mathematica notebook.

\subsection{Recursively generating compositions}\label{recgencompsection}
As stated earlier every composition $\eta$ can be recursively generated from the all zero composition $(0,\ldots,0)$ using a (non-unique) sequence of raising and switching operators. For example, $\eta=(0,2,1)$ can be generated from $(0,0,0)$ by $s_1s_2s_1s_2 \Phi s_2 \Phi s_2 \Phi (0,0,0)$ or more efficiently by $s_2 \Phi s_1 \Phi \Phi(0,0,0)$. In polynomial classes for which switching and raising operators exist analogous methods can be used to generate polynomials. Consequently once we construct an algorithm that generates any composition recursively we automatically obtain an algorithm for the polynomials.

Our aim is to construct an algorithm to recursively generate any composition $\eta$ from $(0,\ldots,0)$ using the least number of operators. We first observe that the raising operator $\Phi$ must be used $|\eta|$ times. Since the raising operator acts on a composition by increasing the value of the component in the first position by one, appending it to the end of the composition and shifting each other component back one position, to minimise the number of operators we must always increase the value of the leftmost component requiring raising. A systematic way of doing this is to apply the raising operator to build all components greater or equal to a specific size only using the switchings to move the leftmost component needing raising to the first position. We note that this method is quite similar to the Rodrigues formulas construction \cite{nishinobc}. Using this method we naturally construct the composition $(\eta^+)^R$, where
$$
\eta^R:=(\eta_n,\eta_{n-1},\ldots,\eta_1).
$$ Due to the nature of $\Phi$ there is no possible way to construct a composition containing each component of $\eta$ using fewer operators.

To reorder $(\eta^+)^R$ minimally we switch each component into its correct position beginning with either $\eta_1,\eta_2,\ldots$ or $\eta_n,\eta_{n-1},\ldots$. We choose to start with repositioning $\eta_n$. By always choosing the closest component of the unordered composition we ensure that the number of switches is minimal.

\begin{proposition}\label{recgencomp}
Define $l_{\eta,i}:=\#\{\eta_j<i\}$, $g_{\eta,i}:=\#\{\eta_j\geq i\}$ and
\begin{equation*}
\sigma(\eta,i):=(l_{\eta,i-1},\ldots,1,l_{\eta,i-1}+1,\ldots,2,\ldots,l_{\eta,i-1}+g_{\eta,i}-1,\ldots,g_{\eta,i}).
\end{equation*}
Define
\begin{equation*}
r_{\eta,i}:=\left\{
\begin{tabular}{ll}
$\Phi^{g_{\eta,i}}$, & $i=1$\\
$\Phi^{g_{\eta,i}}s_{\sigma(\eta,i)}$, & $i>1,$
\end{tabular}
\right.
\end{equation*}
where $s_{(i_l,\ldots,i_1)}:=s_{i_l}\ldots s_{i_1}$. Define
\begin{equation*}
p_{\eta,i}:=\max\Big\{j  \leq i:\prod_{k=1}^{i-1}s_{\sigma'(\eta,k)}(\eta^+)^R_j=\eta_i\Big\}
\end{equation*}
and
$$s_{\sigma'(\eta,i)}:=\left\{
\begin{tabular}{ll}
$s_{(p_{\eta,i},\ldots,i-1)}$, & $p_{\eta,i}<i$\\
$1$ ,& $p_{\eta,i}=i$.
\end{tabular}
\right.
$$ The minimal length sequence of operators that transforms $(0,\ldots,0)$ to $\eta$ is $$s_{\sigma'(\eta,2)}\ldots s_{\sigma'(\eta,n)}r_{\eta,\max(\eta)}\ldots r_{\eta,1}.$$ That is
\begin{equation}\label{recgencompmainresult}
s_{\sigma'(\eta,2)}\ldots s_{\sigma'(\eta,n)}r_{\eta,\max(\eta)}\ldots r_{\eta,1}(0,\ldots,0)=\eta.
\end{equation}
\end{proposition}
\begin{proof}
We prove (\ref{recgencompmainresult}) in two steps. We first show by induction that
\begin{equation}\label{firststeprecgen}
r_{\eta,\max(\eta)}\ldots r_{\eta,1}(0,\ldots,0)=(\eta^+)^R.
\end{equation}
By the definition of $r_{\eta,1}$ it is clear that $r_{\eta,1}(0,\ldots,0)$ produces a composition of the form $(0,\ldots,0,1,\ldots,1)$ where the number of $1$'s is equal to the number of components of $\eta$ that are greater or equal to $1$. Suppose before applying $r_{\eta,k+1}$ we have generated the correct number of components with value $0,1,\ldots,k-1$, that is we have constructed a composition of the form $((\eta^+)^R_1,\ldots,(\eta^+)^R_j,k\ldots,k)$ where the number of $k$'s is equal to the number of components of $\eta$ that are greater or equal to $k$. Quite obviously $r_{\eta,k+1}((\eta^+)^R_1,\ldots,(\eta^+)^R_j,k\ldots,k)=((\eta^+)^R_1,\ldots,(\eta^+)^R_j,k\ldots,k,k+1,\ldots,k+1)$ where the number of $k$'s equals the number of components of $\eta$ equal to $k$ and the number of $(k+1)$'s equals the number of components greater than or equal to $k+1$. By induction (\ref{firststeprecgen}) holds. The final task is to show that
\begin{equation}\label{2ndsteprecgen}
s_{\sigma'(\eta,2)}\ldots s_{\sigma'(\eta,n)}(\eta^+)^R=\eta.
\end{equation}
This result follows immediately from the definition of $p_{\eta,i}$ as quite clearly $\sigma'(\eta,i)$ successively permutes each $\eta_i$ into the correct position. The fact that the total sequence of operators is of minimal length follows from the action of $\Phi$ and the inability to generate a composition with components $\eta_1,\ldots,\eta_n$ more economically than what is specified by (\ref{firststeprecgen}), the permutation that places each component into its correct position can not be improved either.
\hfill $\square$
\end{proof}

We note that further evidence showing the permutation in (\ref{2ndsteprecgen}) is minimal is the comparison of its length to the minimal permutation $\omega_\eta \omega_{(\eta^+)^R}^{-1}$, where $\omega_\eta$ is the minimal length permutation such that $\omega_\eta^{-1}(\eta)=\eta^+$. Due to the different structures of the permutations we use the computational evidence to support our claim.

To provide additional clarity to the algorithm we show how $(4,1,2,1)$ is generated using the above operators.
\begin{example}\label{example}
We construct the composition $(4,1,2,1)$ recursively from $(0,0,0,0)$. We first construct $(1,1,2,4)$ using the operators $r_{\eta,i}$.
\begin{align*}
r_{\eta,4}r_{\eta,3}r_{\eta,2}r_{\eta,1}(0,0,0,0)&=r_{\eta,4}r_{\eta,3}r_{\eta,2}\Phi^4(0,0,0,0)\\
&=r_{\eta,4}r_{\eta,3}r_{\eta,2}(1,1,1,1)\\
&=r_{\eta,4}r_{\eta,3}\Phi^2(1,1,1,1)\\
&=r_{\eta,4}r_{\eta,3}(1,1,2,2)\\
&=r_{\eta,4}\Phi s_1s_2(1,1,2,2)\\
&=r_{\eta,4}(1,1,2,3)\\
&=\Phi s_1s_2s_3(1,1,2,3)\\
&=(1,1,2,4).
\end{align*}
We complete the generation by permuting each component into its correct position
\begin{align*}
s_{\sigma'(\eta,2)} s_{\sigma'(\eta,3)} s_{\sigma'(\eta,4)}(1,1,2,4)&= s_{\sigma'(\eta,2)} s_{\sigma'(\eta,3)} s_3s_2(1,1,2,4) \\
&= s_{\sigma'(\eta,2)} s_{\sigma'(\eta,3)}(1,2,4,1) \\
&=s_{\sigma'(\eta,2)}s_2(1,2,4,1) \\
&=s_{\sigma'(\eta,2)}(1,4,2,1) \\
&= s_1(1,4,2,1) \\
&=(4,1,2,1).
\end{align*}
\end{example}

We now move onto the major goal of the paper, developing the recursive generation algorithm for the nonsymmetric and interpolation Macdonald polynomials.

\subsection{Recursively generating polynomials}\label{Generating polynomials}

Beyond the nonsymmetric and interpolation Macdonald polynomials there are many other families of polynomials that can be recursively generated via switching and raising type operators. Some examples include nonsymmetric and interpolation Jack polynomials (see e.g. \cite[Chap. 12]{forresterbook}) and the generalised nonsymmetric Hermite and Laguerre polynomials (see e.g. \cite[Chap. 13]{forresterbook}). In this section we show how the algorithm developed in Proposition \ref{recgencomp} can be employed to recursively generate any of the polynomials in these families.

For simplicity in this section we use $F_\eta(z)$ to denote any of the aforementioned families of polynomials. We show how the recursive generation algorithm works in the general setting and then provide a specific example for the nonsymmetric Macdonald polynomials.

To most simply express the sequence of operators required to recursively generate a composition according to Proposition \ref{recgencomp} we use the numbers $1,\ldots, n-1$ to represent the allowable switching operators, $0$ to represent the raising operator and denote the required sequence by ${\rm R}(\eta)$. For Example \ref{example} in the previous section we observe that $${\rm R}((4,1,2,1))=\{0, 0, 0, 0, 0, 0, 2, 1, 0, 3, 2, 1, 0, 2, 3, 2, 1 \}.$$

\begin{proposition}\label{recgenpoly}
Define
\begin{equation}\label{recursivegenop}
{\rm RG}_j({F_{{\eta(j)}}(z),\eta(j)},R(\eta)_j):=\left\{
\begin{tabular}{ll}
$\{F_{s_i {\eta(j)}}(z), s_i {\eta(j)}\},$ & $i=R(\eta)_j=1,\ldots,n-1$\\
$\{F_{\Phi {\eta(j)}}(z), \Phi {\eta(j)}\},$ & $i=R(\eta)_j=0,$
\end{tabular}
\right.
\end{equation}
where $\eta(j)$ represents the composition obtained after $j$ transformations from $(0,\ldots,0)$ to $\eta$ specified by ${\rm R}(\eta)$ and $F_{s_i\eta}(z)$ and $F_{\Phi\eta}(z)$ are obtained from $F_\eta(z)$ using known formulas. With initial input ${\rm RG}_1(1,(0,\ldots,0),{\rm R}(\eta)_1)$ and each subsequent polynomial derived from the previous by entering the newly obtained polynomial and composition along with the next number in ${\rm R}(\eta)$ in ${\rm RG}$ the polynomial $F_\eta(z)$ will be obtained after $|R(\eta)|$ steps.
\end{proposition}
\begin{proof}
By Proposition \ref{recgencomp} we know that the sequence specified by $R(\eta)$ will recursively generate the composition $\eta$ from $(0,\ldots,0)$. Consequently $RG_j$ will recursively generate $F_\eta(z)$ from $F_{(0,\ldots,0)}(z)$.
\hfill $\square$
\end{proof}

We note that we keep track of the composition labelling the polynomial at each stage due to requirements of the formulas transforming $F_\eta(z)$ to $F_{s_i\eta}(z)$ and $F_{\Phi\eta}(z)$.

\begin{example}
To recursively generate a nonsymmetric Macdonald polynomial using the methods in the previous proposition we begin by rewriting the recursive generation formulas for the nonsymmetric Macdonald polynomials as
\begin{align}
E_{s_{i}\eta }(z)=&\left\{
\begin{tabular}{ll}
$t^{-1}T_{i}E_{\eta }(z)- \frac{t-1}{t(1-\delta _{i,\eta }^{-1}( q,t)) }
E_{\eta }(z)$, & $\eta _{i}<\eta _{i+1}$ \\
$E_{\eta }(z)$, & $\eta _{i}=\eta _{i+1}$ \\
$\frac{ ( 1-\delta
_{i,\eta }( q,t) ) ^{2}}{( 1-t\delta _{i,\eta }( q,t) ) (
1-t^{-1}\delta _{i,\eta }( q,t) )} \Big(T_iE_{\eta }(z) \frac{t-1}{1-\delta _{i,\eta }^{-1}( q,t) }
E_{\eta }(z)\Big)$, & $\eta _{i}>\eta
_{i+1}$,
\end{tabular}
\right.  \notag \\
E_{\Phi \eta } (z)=&t^{\#\{i>1:\eta_i\leq \eta_1\} }\Phi _{q}E_{\eta }(z). \notag
\end{align}
To generate the polynomial $E_{(2,1)}(z;q,t)$ from $E_{(0,0)}(z;q,t)=1$ using RG we first compute R$((2,1))$ using Proposition \ref{recgencomp} to be R$(2,1)=\{0,0,0,1\}$ we then proceed recursively
\begin{align*}
{\rm RG}_1(1,(0,0),0)&=\{E_{(0,1)}(z),(0,1)\} \\
{\rm RG}_2(E_{(0,1)}(z),(0,1),0)&=\{E_{(1,1)}(z),(1,1)\} \\
{\rm RG}_3(E_{(1,1)}(z),(1,1),0)&=\{E_{(1,2)}(z),(1,2)\} \\
{\rm RG}_4(E_{(1,2)}(z),(1,2),1)&=\{E_{(2,1)}(z),(2,1)\}. \\
\end{align*}
\end{example}

In the cases where $F_\eta(z)$ is homogeneous, for example the nonsymmetric Macdonald and Jack polynomials, we can greatly reduce the number of operators required to generate polynomials labelled by compositions with no zero components by making use of the relationship \cite{marshallthesis}
\begin{equation}\label{makequicker}
F_{\eta+(k^n)}(z)=(z_1\ldots z_n)^kF_\eta(z).
\end{equation}
This result allows us to omit the first $n\times \min\{\eta\}$ zeros from ${\rm R}(\eta)$, forming a new set R$'(\eta)$, and begin our recursive process with $$\{(z_1\ldots z_n)^{\min\{\eta\}}, (\min\{ \eta\},\ldots,\min\{ \eta\}),{\rm R}'(\eta)_{1}\}$$ rather than $\{1,(0,\ldots,0),{\rm R}(\eta)_1\}$.

\begin{example}
Using (\ref{makequicker}) we recursively generate $E_{(2,1)}(z;q,t)$ from $E_{(1,1)}(z;q,t)$. With R$'((2,1))=\{0,1\}$ we obtain
\begin{align*}
{\rm RG}_1(z_1z_2,(1,1),0)&=\{E_{(1,2)}(z),(1,2)\} \\
{\rm RG}_2(E_{(1,2)}(z),(1,2),1)&=\{E_{(2,1)}(z),(2,1)\}. \\
\end{align*}
\end{example}

\subsection{Runtimes and Software}\label{Mathematicanotebook}
\subsubsection{Algorithm runtimes}

In this section we analyse the performance of the algorithm constructed in Proposition \ref{recgenpoly} when used to generate the nonsymmetric and interpolation Macdonald polynomials.

To provide a thorough analysis we must select polynomials of varying degrees of complexity. That is, varying the number of variables, the maximum degree and of these polynomials selecting those requiring the least and most number of operators to generate. Of the polynomials with $n$ variables and maximum degree $k$ those that take the least and most number of operators to generate using the algorithm of Proposition \ref{recgenpoly} are labelled by compositions $(q^{n-r},(q+1)^r)$ and $(k,0,\ldots,0)$, where $k=qn+r$, and require $k$ and $nk$ operators to generate, respectively. We note that we have employed (\ref{makequicker}) into our algorithm for the nonsymmetric Macdonald polynomials and polynomials labelled by compositions with no zero components will be generated more efficiently than what is specified above, for example the nonsymmetric Macdonald polynomials labelled by $\eta=(k^n)$ would be generated almost instantly.

Table \ref{runtimerod} shows the runtimes of the generation of different nonsymmetric Macdonald polynomials using the recursive generation algorithm of Proposition \ref{recgenpoly}. The computer these runtimes were observed on was an iMac 2.4GHz Intel core $2$ duo processor in version $7.01.0$ of Mathematica.
\begin{table}[h!]
\begin{center}
  \begin{tabular}{ c c |c c|c c  }
      \hline
      &&  \multicolumn{2}{c}{$E_\eta(z)$ }&\multicolumn{2}{|c}{$E^*_\eta(z)$  } \\
    \hline
      $|\eta|$ & $\eta$ & time& $\#$ operators & time& $\#$ operators\\ \hline
   $4$ & $(0,4)$ & 0.0308 &7& 0.7792&7 \\
      & $(1,3)$ & 0.0064 &3&  0.0759&5 \\
  & $(2,2)$ & 0.0011 &0 & 0.0009 &4\\
      & $(3,1)$ & 0.0177& 4& 0.6081 &6 \\
    & $(4,0)$ & 0.0674  &8& 12.3216 &8\\
  & $(0,0,4)$ & 0.48 &10& 4.54&10 \\
  & $(1,1,2)$ & 0.0014 &1& 0.0010&4  \\
    & $(2,1,1)$ & 0.0105 &3 & 0.3108 &6\\
        & $(4,0,0)$ & 0.6337 &12 & 65.7861&12 \\
        & $(1,3,0)$ & 0.3171 &9& 30.3733  &9\\
   \hline
  $7$ & $(0,7)$ & 0.2677&13 & 34.4630&13  \\
   & $(3,4)$ & 0.0019&1 & 0.001648&7  \\
  & $(4,3)$ & 0.0070 &2& 0.1471 &8\\
    & $(7,0)$ & 0.4744 &14& 68.2642&14 \\
   \hline

  \end{tabular}
  \caption{Runtimes in seconds and number of operators required for the recursive generation algorithm} \label{runtimerod}
\end{center}
\end{table}

We note that it is the inhomogeniety of the interpolation polynomials that makes their construction time longer than the nonsymmetric polynomials. Furthermore, we observe occasions where fewer operators result in longer runtimes, for example $E_{(0,0,4)}(z)$ compared with $E_{(1,3,0)}(z)$. This can be explained by the complexity of the polynomials generated in the recursive generation process.

\clearpage
\subsubsection{Mathematica notebook}

We now provide the details of the Mathematica notebook containing the polynomials and operators contained within this paper, and more generally the author's PhD thesis \cite{barattathesis}. The Mathematica notebook, titled SpecialFunctions.nb, can be found
at www.ms.unimelb.edu.au/~wbaratta/index.html and ran on Mathematica 7. The original purpose of the Mathematica notebook was to efficiently generate nonsymmetric and nonsymmetric interpolation Macdonald polynomials to develop an understanding of their known theory. Surpassing this motivation it was extensively used throughout the author's PhD candidature to assist with conjecture formulation and testing. We now present a table containing the key functions defined in the Mathematica notebook.
The notations used in the table are consistent with those used throughout the paper. For those extending beyond the current paper (denoted by a $^*$), we refer the reader to \cite{barattathesis}. Note, some functions are not introduced until the following section.

\begin{table}[h!]
\begin{tabular}{ll}
 {\bf Syntax} & {\bf Description} \\\hline
CompositionModulus[$\eta$]&Computes $|\eta|$\\\hline
 SwitchComposition[$\eta,i$]& Computes $s_i\eta$\\
  RaiseComposition[$\eta$]&Computes $\Phi \eta$\\
$^*$cI[$\eta,I$]&Computes $c_I(\eta)$\\\hline
 $^*$Dominance[$\eta,\lambda$]& Determines whether $\eta\leq\lambda$ or $\lambda\leq\eta$\\
 $^*$PartialOrder[$\eta,\lambda$] & Determines whether $\eta\preceq\lambda$ or $\lambda\preceq\eta$\\
  $^*$PartialOrder2[$\eta,\lambda$] & Determines whether $\eta \triangleleft\lambda$ or $\lambda \triangleleft\eta$\\
 $^*$Successor[$\eta,\lambda$]& Determines whether $\eta\preceq'\lambda$ or $\lambda\preceq'\eta$  \\\hline
  $^*$ArmLength[$\eta,i,j$] & Computes $a_\eta(i,j)$, the armlength of a composition at a square\\
  $^*$ArmCoLength[$\eta,i,j$] & Computes $a'_\eta(i,j)$, the coarmlength of a composition at a square\\
  $^*$LegLength[$\eta,i,j$] & Computes $l_\eta(i,j)$, the leglength of a composition at a square\\
    LegCoLength[$\eta,i$] & Computes $l'_\eta(i)$, the coleglength of a composition at a square\\\hline
    $^*$Md[$\eta$] & Computes $d_\eta(q,t)$\\
      $^*$MdDash[$\eta$] & Computes $d'_\eta(q,t)$\\
        $^*$Me[$\eta$] & Computes $e_\eta(q,t)$\\
          $^*$MeDash[$\eta$] & Computes $e'_\eta(q,t)$\\\hline
          R[$\eta$]& Computes the set R$(\eta)$, specifying the sequence of operators \\
          &required to generate $\eta$ from $(0,\ldots,0)$\\\hline
                  \end{tabular}
\caption{Mathematica functions relating to compositions}\label{Mathematicafunctionsc}
\end{table}

\vspace{0.5cm}

\begin{table}[h]
\begin{tabular}{ll}
 {\bf Syntax} & {\bf Description} \\\hline
PermutationOnComposition[$\sigma,\eta$]& Computes $\sigma(\eta)$\\
 DecompositionOnComposition[$\{s_{i_1},\ldots,s_{i_l}\},\eta$]& Computes $s_{i_l}\ldots s_{i_1}\eta$\\
  SwitchingOperator[$f,i,j$] & Computes $s_{ij}f(\ldots,z_i,\ldots,z_j,\ldots)$\\
  PermutationOnPolynomial[$\sigma,f$]& Computes $\sigma f(z_1,\ldots,z_n)$\\
 ShortestPermutation[$\eta$] & Computes $\omega_\eta$, the shortest \\ &permutation such that $\omega_\eta^{-1}(\eta)=\eta^+$\\
 RequiredPermutation[$\eta,\lambda$] & Computes the permutation $\sigma$ such\\ & that $\sigma(\eta)=\lambda$ \\\hline
          \end{tabular}
\caption{Mathematica functions relating to permutations}\label{Mathematicafunctionsp}
\end{table}

\begin{table}[h]
\begin{tabular}{ll}
 {\bf Syntax} & {\bf Description} \\\hline
         $^*$ Monomial[$\eta$] & Computes $z^\eta$\\
          ElementarySymmetricFunction[$r,n$]& Computes $e_r(z)$ in $n$ variables\\
          $^*$CompleteSymmetricFunction[$r,n$]& Computes $h_r(z)$ in $n$ variables\\
SymmetricMonomialFunction[$\kappa$] & Computes $m_\kappa(z)$\\\hline
$^*$Vandermonde[$n$]& Computes $\Delta(z)$\\
$^*$VandermondeJ[$J$]& Computes $\Delta^J(z)$\\
tVandermonde[$n$]& Computes $\Delta_t(z)$\\
$^*$tVandermondeJ[$J$]&Computes $\Delta^J_t(z)$\\\hline
Schur[$\kappa$] & Computes $s_\kappa(z)$\\
Zonal[$\kappa$] & Computes $Z_\kappa(z)$\\
HallLittlewood[$\kappa$] & Computes $P_\kappa(z;t)$\\\hline
\end{tabular}
\caption{Mathematica functions relating to miscellaneous polynomials}\label{MathematicafunctionsMisc}
\end{table}

\begin{table}[h]
\begin{tabular}{ll}
 {\bf Syntax} & {\bf Description} \\\hline
 SymJack[$\kappa$]& Computes $P_\kappa(z;\alpha)$\\
 NSJack[$\eta$]& Computes $E_\eta(z;\alpha)$\\
IntJack[$\eta$]& Computes $E^*_\eta(z;\alpha)$\\\hline
$^*$JEvalue[$\eta$] & Computes $\overline{\eta}^\alpha$, where $\overline{\eta}_i:=\alpha \eta_i - l'_\eta(i)$\\\hline
$^*$EOpSymJack[$f,n$] &Computes $D_2(\alpha)f$\\
$^*$EOpNSJack[$f,n,i$] &Computes $\xi_i f$\\
$^*$EOpIntJack[$f,n,i$] &Computes $\Xi^\alpha_i f$\\\hline
$^*$JInnerProduct[$f,g,n,k$]& Computes $\langle f,g\rangle_{1/k}$ for polynomials $f,g$ of $n$ variables\\\hline
\end{tabular}
\caption{Mathematica functions relating to Jack polynomials}\label{MathematicafunctionsJack}
\end{table}

\begin{table}[h!]
\begin{tabular}{ll}
 {\bf Syntax} & {\bf Description} \\\hline
 NSMac[$\eta$]& Computes $E_\eta(z;q,t)$\\
SymMac[$\kappa$]& Computes $P_\kappa(z;q,t)$\\
ASymMac[$\kappa$]& Computes $S_\kappa(z;q,t)$\\\hline
IntMac[$\eta$]& Computes $E^*_\eta(z;q,t)$\\
SymIntMac[$\kappa$]& Computes $P^*_\kappa(z;q,t)$\\
ASymIntMac[$\kappa$]& Computes $S^*_\kappa(z;q,t)$\\\hline
  MEvalue[$\eta$] & Computes $\overline{\eta}$, where $\overline{\eta}_i:=q^{\eta_i}t^{l'_\eta(i)}$\\\hline
         $^*$EOpSymMac[$f,n$] &Computes $D_n^1(q,t)f$\\
$^*$EOpNSMac[$f,i$] & Computes $Y_if$\\
$^*$EOpIntMac[$f,i$] & Computes $\Xi_if$\\\hline
Ti[$f,i$]& Computes $T_if$\\
TiInv[$f,i$]&Computes $T_i^{-1}f$\\
Phiq[$f,n$]&Computes $\Phi_qf$\\
Hi[$f,i$]&Computes $H_if$\\
PhiqInt[$f,n$]&Computes $\Phi^*_qf$\\\hline
$^*$MInnerProduct[$f,g,n,k$]& Computes $\langle f,g\rangle_{q,q^k}$, for polynomials $f,g$ of $n$ variables\\\hline
\end{tabular}
\caption{Mathematica functions relating to Macdonald polynomials}\label{MathematicafunctionsMac}
\end{table}

\begin{table}[h!]
\begin{tabular}{ll}
 {\bf Syntax} & {\bf Description} \\\hline
$^*$UPlus[$f,n$] & Computes $U^+f$, where $f$ is a function of $n$ variables\\
$^*$UMinus[$f,n$] & Computes $U^-f$, where $f$ is a function of $n$ variables\\
$^*$UPlusInt[$f,n$] & Computes $U^+_*f$, where $f$ is a function of $n$ variables\\
$^*$UMinusInt[$f,n$] & Computes $U^-_*f$, where $f$ is a function of $n$ variables\\\hline
 OIJ[$f,n,I,J$]& Computes $O_{I,J}f$, where $f$ is a function of $n$ variables\\
OIJInt[$f,n,I,J$]& Computes $O^*_{I,J}f$, where $f$ is a function of $n$ variables\\\hline
PreSymMac[$\eta^*,I,J$]& Computes $S_{\eta^*}^{(I,J)}(z;q,t)$\\
PreSymIntMac[$\eta^*,I,J$]& Computes $S_{\eta^*}^{*,(I,J)}(z;q,t)$\\\hline
aEta[$\eta,I,J$]& Computes the normalisation $a_\eta^{(I,J)}(q,t)$ \\\hline
\end{tabular}
\caption{Mathematica functions relating to prescribed symmetry polynomials}\label{MathematicafunctionsSym}
\end{table}
\textcolor{white}\\
\begin{table}[h]
\begin{tabular}{ll}
 {\bf Syntax} & {\bf Description} \\\hline
 $^*$JEta[$\eta,r$] & Produces the set of $\lambda$ such that $|\lambda|=|\eta|+r$\\and  &  $\eta\preceq'\lambda\preceq'\eta+(1^n)$\\\hline
 Pieri[$\eta,\lambda,r$] & Computes the Pieri-type coefficient $A^{(r)}_{\eta,\lambda}(q,t)$\\\hline
 \T $^*$GeneralisedBinomialCoefficient[$\lambda,\eta$] & Computes $\binom{\lambda}{\eta}_{q,t}$\\\hline
\end{tabular}
\caption{Mathematica functions relating to Pieri-type coefficients}\label{MathematicafunctionsPieri}
\end{table}

\vspace{1cm}
\begin{table}[h]
\begin{tabular}{p{14cm}}
{\small {\bf Copyleft} Copyleft 2011 Wendy Baratta}\\
{\small Permission is granted to anyone to use, modify and redistribute SpecialFunction.nb freely subject to the following.}\\
{\small$\cdot$ We make no guarantees that the software is free of defects.} \\
{\small$\cdot$ We accept no responsibility for the consequences of using this software.}\\
{\small$\cdot$ All explicit use of this notebook must be explicitly represented.}\\
{\small$\cdot$ No form of this software may be included or redistributed in a library to be sold for profit without our consent.}\end{tabular}
\end{table}

\vskip 2cm

\clearpage

\section{Further Work}\label{furtherworkMathematica}

It is hoped that the notebook SpecialFunctions.nb will continue to assist researchers with conjecture formulation and testing in areas relating to the polynomial families discussed within this paper. In fact this has already shown itself to be the case in a study of special vanishing properties of Jack polynomials for $\alpha=-(r-1)/(k+1)$ and Macdonald polynomials with $t^{k+1}q^{r-1}=1$ \cite{PFWB}. Below are some examples of computations that aim to provide motivation for two research problems stemming from the author's previous studies; Macdonald polynomials with prescribed symmetry \cite{barattapresym} and Pieri-type formulas for nonsymmetric Macdonald polynomials \cite{barattafurtherpieri}.

\subsection{Interpolation polynomials with prescribed symmetry}
In \cite{barattapresym} the author investigated properties of Macdonald polynomials with prescribed symmetry. Macdonald polynomials with prescribed symmetry, denoted $S_{\eta^*}^{(I,J)}(z;q,t)$, are  generalisations of the symmetric Macdonald polynomials and the antisymmetric Macdonald polynomials, denoted $S_{\eta^++\delta}(z;q,t)$ where $\delta=(n-1,\ldots,1,0)$, and are symmetric with respect to some variables and antisymmetric with respect to others. They are obtained from the nonsymmetric Macdonald
polynomials via symmetrisation by $O_{I,J}$ and normalisation. The subscript $I$, $J$  indicates
the sets of variables for which $O_{I,J}$ symmetrises and antisymmetrises with respect to. The operator $O_{I,J}$ is defined by
$$
O_{I,J}:=\sum_{\omega\in W_{I\cup J}}\Big(-\frac{1}{t} \Big)^{l(\omega_J)}T_\omega,
$$
where $l(\omega)$ is the length of the permutation $\omega$,
\begin{equation}\label{compositionofop}T_\omega:=T_{i_1}\ldots T_{i_l},\end{equation}
where $s_{i_1}\ldots s_{i_l}$ is a reduced decomposition of $\omega$, and $W_{I\cup J}:=\langle s_k; k\in I\cup J\rangle$, a subset of $S_n$ where each $\omega\in W_{I\cup J}$ can be decomposed as
$$
\omega=\omega_I\omega_J,\hspace{1cm} \text{with }\omega_I\in W_I \text{ and } \omega_J\in W_J.
$$

The prescribed symmetric Macdonald polynomials is obtained from the nonsymmetric Macdonald polynomials according to
$$
O_{I,J}E_\eta(z;q,t)=a_\eta^{(I,J)}S_{\eta^*}^{(I,J)}(z;q,t),
$$
where $\eta^*$ is a composition satisfying
$$
\eta_i^*\geq \eta_{i+1}^* \text{ for all }i\in I \text{ and } \eta_j^*>\eta_{j+1}^* \text{ for all } j\in J
$$ and $a_\eta^{(I,J)}$ is a normalisation that ensures the coefficient of $z^{\eta^*}:=z_1^{\eta_1^*}\ldots z_n^{\eta_n^*}$ is unity.

The theory of Macdonald polynomials with prescribed symmetry leads most naturally to an investigation of interpolation polynomials with prescribed symmetry. Since the operator $H_i$ (\ref{TiHi}) satisfies the Hecke algebra (\ref{hecke algebra}) and plays the same role in interpolation theory as $T_i$ does in nonsymmetric theory it seems natural to define the prescribed symmetry operator for the interpolation polynomials as
$$
O^*_{I,J}:=\sum_{\omega\in W_{I\cup J}}\Big(-\frac{1}{t} \Big)^{l(\omega_J)}H_\omega,
$$
where $H_\omega$ is as in (\ref{compositionofop}).

Trial computation suggest that it may be possible to extend many results obtained in \cite{barattapresym} to their interpolation polynomial analogues (see for example, Propositions 3.2.1, 3.2.2 and 3.3.1).

Specifically, Proposition 3.3.1 in \cite{barattapresym} specifies the relationship between the antisymmetric and symmetric Macdonald polynomial
$$
S_{\eta^++\delta}(z;q,t)=\Delta_t(z)P_{\eta^+}(z;q,qt),
$$
where $\Delta_t(z)$ is the $t$-Vandermonde and specified by
$$
\Delta_t(z):=\prod_{1\leq i<j \leq n}(z_i-t^{-1}z^j).
$$
This property has extensions in the theory of Macdonald polynomials with prescribed symmetry and is therefore a suitable starting point for the theory of interpolation Macdonald polynomials with prescribed symmetry.

Here we provide some explicit computations obtained using the Mathematica notebook that may assist with the identification of the relationship between the antisymmetric and symmetric interpolation polynomial.

\textcolor{white}{space}\\
\begin{tabular}{ll}
{\ttfamily In[1] }&{\ttfamily SymIntMac[$\{1,0\}$]}\\
{\ttfamily Out[1]}& \TT\BB $\dfrac{t \left(x_1+x_2\right)-t-1}{t}$ \\
{\ttfamily In[2] }&{\ttfamily ASymIntMac[$\{2,0\}$]}\\
{\ttfamily Out[2]}&  \TT\BB$\dfrac{\left(x_1-t x_2\right) \left(t \left(x_1+x_2\right)-q t-1\right)}{t}$
\end{tabular}

\textcolor{white}{space}\\
\begin{tabular}{ll}
{\ttfamily In[1] }&{\ttfamily SymIntMac[$\{1,1\}$]}\\
{\ttfamily Out[1]}& \TT\BB $ \dfrac{\left(t x_1-1\right) \left(t x_2-1\right)}{t^2}$ \\
{\ttfamily In[2] }&{\ttfamily ASymIntMac[$\{2,1\}$]}\\
{\ttfamily Out[2]}&  \TT\BB$\dfrac{ \left(x_1-t x_2\right)\left(t x_1-1\right)  \left(t x_2-1\right)}{t^2}$
\end{tabular}

\textcolor{white}{space}\\
\begin{tabular}{ll}
{\ttfamily In[1] }&{\ttfamily SymIntMac[$\{1,0,0\}$]}\\
{\ttfamily Out[1]}& \TT\BB $ \dfrac{t^2 \left(x_1+x_2+x_3\right)-t^2-t-1}{t^2}$ \\
{\ttfamily In[2] }&{\ttfamily ASymIntMac[$\{3,1,0\}$]}\\
{\ttfamily Out[2]}&  \TT\BB$\dfrac{\left(x_1-t x_2\right) \left(x_1-t x_3\right) \left(x_2-t x_3\right) \left(t^2 \left(x_1+x_2+x_3\right)-q t (q t+1)-1\right)}{t^2}$
\end{tabular}

\textcolor{white}{space}\\
These computations show a clear relationship between $S^*_{\kappa+\delta}(z)$ and $P^*_\kappa(z)$, highlighting that the difficulty in identifying the relationship lies in the transformations of the parameters. It also appears that the $t$-Vandermonde product that relates the polynomials is not of standard form.

\subsection{Pieri-type formulas for nonsymmetric Macdonald polynomials}
In \cite{barattafurtherpieri} the formulas for the coefficients in the expansion of a nonsymmetric Macdonald polynomial and an elementary symmetric function were given. Explicitly, the formulas for $A_{\eta\lambda}^{(r)}$ in
$$
e_r(z)E_\eta(z;q^{-1},t^{-1})=\sum_{\lambda:|\lambda|=|\eta|+r}A_{\eta\lambda}^{(r)}E_\lambda(z;q^{-1},t^{-1}),
$$
where
$$
e_r(z):=\sum_{1\leq i_1<\ldots <i_r\leq n}z_{i_1}\ldots z_{i_n}.
$$
The formulas of $A_{\eta\lambda}^{(r)}$ given in \cite{barattafurtherpieri} are given in summation form, though computational evidence suggests that almost all of the coefficients display a product structure. Here we present explicit formulas for the Pieri-type coefficients for the case $r=2$ showing which cases display a product structure and which cases do not.

\textcolor{white}{space}\\
\begin{tabular}{rlcrl}
{\ttfamily\small{ In[1]}\TT }&{\ttfamily \small{Pieri[}$\tiny{\{1,0,1,0\},\{2,0,2,0\},2}$\small{]}}&\textcolor{white}{ee}& {\ttfamily \small{In[2]} }&{\ttfamily \small{Pieri[}$\tiny{\{1,0,1,0\},\{2,1,1,0\},2}$\small{]}}\\
{\ttfamily \small{Out[1]}}& \TT\BB $\tiny{1}$ && {\ttfamily \small{Out[2]}}& \TT\BB $\tiny{\dfrac{q - 1} {q t - 1}}$
\end{tabular}

\textcolor{white}{space}\\
\begin{tabular}{rlcrl}
{\ttfamily\small{ In[3]}\TT }&{\ttfamily \small{Pieri[}$\tiny{\{1,0,1,0\},\{1,1,2,0\},2}$\small{]}}&\textcolor{white}{ee}& {\ttfamily \small{In[4]} }&{\ttfamily \small{Pieri[}$\tiny{\{1,0,1,0\},\{2,0,1,1\},2}$\small{]}}\\{\ttfamily \small{Out[3]}}& \TT\BB $\tiny{\dfrac{t (q - 1) \left (q t^3 - 1 \right)} {\left (q t^2 - 1 \right)^2}}$ && {\ttfamily \small{Out[4]}}& \TT\BB $\tiny{\dfrac{t(q - 1)  \left (q t^3 - 1 \right)} {\left (q t^2 - 1 \right)^2}}$
\end{tabular}

\textcolor{white}{space}\\
\begin{tabular}{rlcrl}
{\ttfamily\small{ In[5]}\TT }&{\ttfamily \small{Pieri[}$\tiny{\{1,0,1,0\},\{1,1,1,1\},2}$\small{]}}&\textcolor{white}{ee}& {\ttfamily \small{In[6]} }&{\ttfamily \small{Pieri[}$\tiny{\{1,0,1,0\},\{1,2,0,1\},2}$\small{]}}\\{\ttfamily \small{Out[5]}}& \TT\BB $\tiny{\dfrac{(q - 1) (q t - 1) \left (q t^4 -
   1 \right)} {\left (q t^2 - 1 \right)^3}}$ && {\ttfamily \small{Out[6]}}& \TT\BB $\tiny{\dfrac{q^2t^2(q - 1)  (t - 1)^4 } {(q t - 1)^3 \left (q t^2 - 1 \right)^2}}$
\end{tabular}

\textcolor{white}{space}\\
\begin{tabular}{rlcrl}
{\ttfamily\small{ In[7]}\TT }&{\ttfamily \small{Pieri[}$\tiny{\{1,0,1,0\},\{1,2,1,0\},2}$\small{]}}&\textcolor{white}{ee}& {\ttfamily \small{In[8]} }&{\ttfamily \small{Pieri[}$\tiny{\{1,0,1,0\},\{2,1,0,1\},2}$\small{]}}\\
{\ttfamily \small{Out[7]}}& \TT\BB $\scriptsize{-\dfrac{qt(q - 1)  (t - 1)^2 } {(q t - 1)^2 \left (q t^2 - 1 \right)}}$ && {\ttfamily \small{Out[8]}}& \TT\BB $\tiny{-\dfrac{qt(q - 1)  (t - 1)^2 } {(q t - 1)^2 \left (q t^2 - 1 \right)}}$
\end{tabular}

\textcolor{white}{space}\\
\begin{tabular}{rlcrl}
{\ttfamily\small{ In[9]}\TT }&{\ttfamily \small{Pieri[}$\tiny{\{1,0,1,0\},\{0,2,2,0\},2}$\small{]}}&\textcolor{white}{t}& {\ttfamily \small{In[10]} }&{\ttfamily \small{Pieri[}$\tiny{\{1,0,1,0\},\{2,0,0,2\},2}$\small{]}}\\
{\ttfamily \small{Out[9]}}& \TT\BB $\tiny{\dfrac{qt^2(q - 1)  (t - 1) } {(q t - 1) (q t + 1) \left (q t^2 -
  1 \right)}}$ && {\ttfamily \small{Out[10]}}& \TT\BB $\tiny{\dfrac{ qt^2(q - 1) (t - 1) } {(q t - 1) (q t + 1) \left (q t^2 -
    1 \right)}}$
\end{tabular}

\textcolor{white}{space}\\
\begin{tabular}{rlcrl}
{\ttfamily \small{In[11]}\TT }&{\ttfamily \small{Pieri[}$\tiny{\{1,0,1,0\},\{1,0,2,1\},2}$\small{]}}&& &\\
{\ttfamily \small{Out[11]}}& \TT\BB $\small{\dfrac{(q-1) t^2 \left(q \left(t \left(q \left((q-2) q t^5+(3 q-1) t^4+(2-3 q) t^3-2t+1\right)+3 (t-1)\right)+2\right)-1\right)}{(q t-1) \left(q t^2-1\right)^2 \left(q^2t^3-1\right)}}$ && & \\
\end{tabular}

\textcolor{white}{space}\\
In the above formulas for the Pieri-type coefficients $A^{(2)}_{(1,0,1,0),\lambda}(q,t)$ a simple product structure in $q$ and $t$, with roots in $t$ being simple fractional powers of $q$ for example, is exhibited in all cases except $\lambda=(1,0,2,1)$. Additional trial computations suggest that it is always the successor of the form $\Phi\Phi\eta$ that cannot be expressed as a simple product.

\section*{Acknowledgement}
 This work was supported by an Australian Mathematical Society Lift-off Fellowship. I would like to thank AustMS for providing the funding that allowed me to write this paper and further the work from my PhD. I would also like to thank my PhD supervisor who guided me through the research contained in this paper, and kindly agreed to proof read the work after his days as my supervisor were completed. Additional thanks goes to my friend Sam Blake who taught me how to use mathematica.

\bibliographystyle{plain}

\end{document}